\documentclass[12pt]{article}

\usepackage{amsmath,amsthm,bbm}
\usepackage{amssymb,amsfonts}
\usepackage{graphicx}
\usepackage{datetime}
\usepackage{enumerate}
\usepackage{float}
\usepackage[numbers,square]{natbib}
\usepackage{hyperref}
\usepackage{geometry}

\graphicspath{{./Figures/}}

\newtheorem{theorem}{Theorem}[section]

\newtheorem{lemma}[theorem]{Lemma}
\newtheorem{proposition}[theorem]{Proposition}

\newtheorem{assumption}[theorem]{Assumption}
\newtheorem{remark}[theorem]{Remark}

\usepackage[dvipsnames]{xcolor}

\DeclareMathOperator*{\argmax}{arg\,max}

\title{Resource Extraction from a Stochastic Population with Harvesting Quotas}

\author{
Erik Ekström\thanks{Department of Mathematics, Uppsala University, Box 256, 751 05 Uppsala. Email: ekstrom@math.uu.se.}, 
Dante Mata\thanks{D\'epartement de Math\'ematiques, Universit\'e du Qu\'ebec \`a Montr\'eal (UQAM), 201 av. Pr\'esident-Kennedy, Montr\'eal, QC H2X 3Y7, Canada. Email: mata\_lopez.dante@uqam.ca} 
and Harto Saarinen\thanks{Department of Economics, Turku School of Economics,
FIN-20014 University of Turku, Finland. Email: hoasaa@utu.fi.}}
\date{}

\begin{document}

\maketitle

\begin{abstract}
We study a resource extraction problem of ergodic type, where 
harvesting quotas are imposed that restrict the harvesting rate by a function of the current biomass. In a setting with diffusive dynamics, we provide conditions under which the optimal harvesting strategy is of threshold type. In contrast with the unrestricted case with singular controls, the optimal threshold may be 0, which corresponds to harvesting at the maximal possible rate at all times. Furthermore, we establish the convergence of the solution to the corresponding singular control problem as the harvesting rate tends to infinity.
\end{abstract}



\section{Introduction}

The optimal harvesting problem plays a central role in the management of renewable natural resources, with applications to the fishery and forestry industries, among others. In many models 
the evolution of a biological population is driven by a stochastic growth dynamics, and the main question is to identify harvesting strategies that maximise a given objective, such as the expected discounted profit or the  expected average yield. Early studies (see e.g.~\cite{Gleit1978, Ludwig1980}) already address the question of optimal harvesting of populations and obtain  strategies that maximise an expected discounted profit. 
For a comprehensive economic treatment of resource extraction, including the trade-off between exploitation and sustainability, we refer to the monograph of Clark~\cite{Clark1976}.

A substantial body of literature has established that, in many problem formulations and under various assumptions regarding the dynamics of the underlying population, optimal harvesting policies often take a simple form, namely that of a {\em threshold} strategy. 
Under such a strategy, harvesting is deployed at maximal intensity once the population size exceeds a critical level, and withheld should the population size fall below it. In~\cite{Alvarez1998, Lungu1997} the optimality of reflection at a threshold has been proven under a discounted criterion and singular control strategies. Related threshold results have been obtained for impulse controls~\cite{LiuZervos2025} and for relaxed controls~\cite{Stockbridge2013}.

Motivated by interests of sustainability, the ergodic criterion has been argued to be more suitable than the discounted criterion since it, by definition, emphasizes the long-term performance. In particular, under the ergodic criterion any harvesting strategy that depletes the population is automatically deemed sub-optimal.
Optimal harvesting under the ergodic criterion has recently been studied by Alvarez and Hening~\cite{AlvarezHening2022} and by Liang et al.~\cite{LiangLiuZervos2025}, in both cases for singular controls. Notably, in \cite{LiangLiuZervos2025}, a running reward is included in the performance criterion, with reference  
to this as a utility of keeping the population level at a certain level.
While both these studies are motivated in part by sustainability considerations, 
one may argue that the trade-off between sustainability and profit maximization can be further developed.
More specifically, in many situations the harvester 
is a pure profit maximizer who is only avoiding over-extraction if it is in their financial interest to do so,
with no or little interest in ecological considerations per se. In fact, the trade-off between profit maximization and sustainability is an inherent conflict between different parties, and, in particular, it is not always reasonable to impose the safeguarding of the environment on the harvester.

It may therefore seem natural to separate the question of profit maximization and sustainability, and at least two complementary approaches are possible:
\begin{enumerate}
    \item Introducing harvesting quotas, where the amount that may be harvested is limited by regulation;
    \item Imposing environmental conditions, where the harvester may extract in any manner, as long as a long-term sustainability criterion is met.
\end{enumerate}
In this paper we follow the first approach and study an ergodic optimal harvesting problem in which the harvesting rate is bounded by a function of the current biomass. 
While this problem is motivated by harvesting quotas, we note that such a restriction on the harvesting rate can also be motivated by physical constraints, such as the vessel capacity for a fishery.
Rate-constrained harvesting has mostly been analysed under the discounted criterion. In \citep{Alvarez1998} the authors treated constant-rate restrictions in the logistic model and showed that the optimal policy remains a threshold rule. More recently, \citep{Ekstrom2023} analyzed a game version of the bounded-rate extraction problem and proved the existence of Nash equilibria in threshold strategies. Closest to the present work, \citep{Guerin2024, LocasRenaud2024, RenaudRochSimard2026} considered a state-dependent bound in a general diffusion model and proved the optimality of a threshold strategy.
An important step was undertaken by Hening et al.~\cite{Hening2019}, who considered an ergodic problem where the harvesting rate is bounded by a constant.
To the best of our knowledge, the current combination — ergodic criterion, stochastic dynamics, and state-dependent harvesting constraints — has not been previously studied, and it gives rise to qualitatively new mathematical challenges as well as richer and more nuanced optimal policies.

Our main contribution is to show that, in the presence of a 
level-dependent harvesting constraint, the optimal strategy remains of threshold type under general conditions and we provide a semi-explicit characterization of the optimal threshold. 
We also demonstrate that, in contrast with the unrestricted setting, the optimal threshold may collapse to zero, corresponding to continuous harvesting at the maximal admissible rate for any positive biomass levels.
Furthermore, as in \cite{Cohen2022} and \cite{LiangLiuZervos2025}, an important difficulty is that solutions to the HJB equation may fail to be bounded from below. 
This requires a modification of the approach in \cite{LiangLiuZervos2025} to accommodate the somewhat different HJB equation arising in our setting. Another contribution of the paper is the proof of monotonic convergence to the singular control problem of \citep{AlvarezHening2022} as the maximal harvesting rate tends to infinity.

The paper is organized as follows. In Section~\ref{2} we introduce the stochastic model and formulate the ergodic harvesting problem in a random environment with a level-dependent constraint on the harvesting rate. In Section~\ref{sec:threshold} we introduce threshold strategies, which are our candidate optimal strategies. Then, in Sections~\ref{b=0} and~\ref{b>0} we derive the associated Hamilton-Jacobi-Bellman equations, and prove the optimality of the threshold strategies. In Section~\ref{Sec:Comparative} we provide a comparative analysis showing the link between our results and the singular control case, as the control rate increases to infinity. We conclude in Section~\ref{Sec:Numerical} by providing numerical examples that illustrate our results.

\section{Problem formulation and main assumptions}
\label{2}
Let $(\Omega, \mathcal{F}, \{\mathcal{F}_t\}_{t \geq 0}, \mathbb{P})$ be a filtered probability space where the filtration $\{\mathcal{F}_t\}_{t \geq 0}$ satisfies the usual conditions of completeness and right-continuity. We consider a linear diffusion $X$ defined on $(\Omega, \mathcal{F}, \{\mathcal{F}_t\}_{t \geq 0}, \mathbb{P})$, which lives on $[0,\infty)$, and which represents the total biomass of a population in the case of no harvesting.
More specifically, we let $X$ satisfy
the stochastic differential equation
\begin{equation} \label{Eq:UncontrolledDiffusion}
    dX_t=\mu(X_t)dt+ \sigma(X_t)dW_t, \quad \quad X_0 = x,
\end{equation}
where $W_t$ is a Wiener process and the functions $\mu: [0,\infty) \to \mathbb{R}$ and $\sigma:[0,\infty) \to [0,\infty)$ are locally Lipschitz continuous and with $\sigma^2(x)>0$ for all $x \in \mathbb{R}_+$. These assumptions guarantee that the stochastic differential equation (\ref{Eq:UncontrolledDiffusion}) has a unique strong solution defined up to a possible explosion time. We also assume that the boundary 
$\infty$ is natural and that the boundary $0$ is either natural or entrance-not-exit (so unattainable), and in particular explosion ($X=\infty$) or extinction ($X=0$) does not happen in finite time. 
For comprehensive characterizations of the boundary behavior of linear diffusions, see \cite[pp. 14-21]{BorodinSalminen} or \cite[Chapter 15.6]{Karlin1981}. 
Even though we consider the case where the process evolves on $(0,\infty)$, we remark that our results would remain the same with obvious changes if the state space would be replaced by another interval.

To introduce our control problem, consider a process $\nu=(\nu_t)_{t\geq 0}$ with $\nu_0=0$ which is nondecreasing, continuous and adapted to $\{\mathcal{F}_t\}_{t \geq 0}$, and let $k:[0,\infty)\to[0,\infty)$ be a given non-decreasing function with $k(x)>0$ for all $x>0$. 
The process $\nu$ is said to be an {\em admissible control} if 
\begin{itemize}
\item the equation 
\begin{equation} \label{Eq:ControlledDiffusion}
    dX^{\nu}_{t} = \mu(X^{\nu}_{t}) dt + \sigma(X^{\nu}_{t}) dW_{t} - d\nu_{t}, \quad X_0 = x,
\end{equation}
admits a unique strong solution on the stochastic interval $[0,\tau^\nu]$, where 
$\tau^\nu=\inf\{t\geq 0:X^\nu_t\leq 0\}$, and 
\item 
$d\nu_t\leq k(X^\nu_t)\mathbf 1_{\{t\leq \tau^\nu\}} dt$ for all $t\geq 0$.
\end{itemize}
We note that no more control can be exerted after the first hitting time of $0$, and we
denote the set of admissible controls $\mathcal{V}$. 

Given an admissible control $\nu\in\mathcal V$, the average reward per unit time (ergodic reward) is defined as
\begin{equation*} 
    J(x; \nu) = \liminf_{T \to \infty} \frac{\mathbb{E}_x [v_T]}{T} .
\end{equation*}
The objective of the ergodic control problem is to maximize $J(x; \nu)$ over all admissible controls, i.e. to find
\begin{equation} \label{Eq:ControlProblemFormulation}
    \sup_{\nu \in \mathcal{V}}  J(x;\nu),
\end{equation}
and to find an optimal control $\nu^*\in\mathcal V$ for which $J(x;\nu^*)=\sup_{\nu \in \mathcal{V}}  J(x,\nu)$.

To formulate our main assumptions, we introduce the scale function $S$ and the speed measure $m$ of the uncontrolled process $X$. In the current setting, we have 
$S'(x) = e^{-B(x)}$ and $m(dx) = m'(x)dx$, where $m'(x)=\frac{2}{\sigma^2(x)} e^{B(x)}$,
$B(x) = \int_{c_0}^x \frac{2\mu(z)}{\sigma^2(z)}dz$ and $c_0 \in \mathbb{R}_+$ is a fixed, but  otherwise arbitrary, constant. 
In calculations below, the identity 
\begin{equation} \label{Eq:ScaleDriftInt}
    \int_x^y \mu(z)m(dz) = \frac{1}{S'(y)}-\frac{1}{S'(x)}
\end{equation}
for $x,y>0$ will be used.

We study the ergodic control problem \eqref{Eq:ControlProblemFormulation} under the following main assumptions (some of which have already been introduced, but are repeated here for clarity and for future reference).

\begin{assumption} \label{Ass:MainAssumptions} 
The coefficients $\mu$ and $\sigma$ are locally Lipschitz continuous with $\sigma^2(x)>0$ for $x>0$, and such that 
    \begin{enumerate}     
            \item[(i)]
            the boundary point $\infty$ is natural, and 0 is natural or entrance-not-exit;
            \item[(ii)]
        there exists an $x^* =\argmax_x \{ \mu(x) \}>0$ so that $\mu(x)$ is increasing on $[0, x^*)$ and decreasing on $(x^*, \infty)$, with
        \[ \mu(0) \geq 0 \quad\&\quad\lim_{x \to \infty} \mu(x) < 0;\]
        \item[(iii)]
        $\int_0^{\infty} m(dy) <\infty$;
        \item[(iv)]
        $S'(\infty):=\lim_{x\to\infty}S'(x)=\infty$.
    \end{enumerate}
    Moreover, 
      \begin{enumerate}     
            \item[(v)]
            $k:[0,\infty)\to[0,\infty)$ is non-decreasing, and $k(x)>0$ for $x>0$.
              \end{enumerate}
\end{assumption}

\begin{remark}
    The assumptions on the drift in (ii) are well aligned with the assumptions used for singular ergodic control problems that investigate barrier strategies for linear diffusion processes, see \citep{Alvarez2018,AlvarezHening2022,Christensen2025}. The assumptions are also analogous to the assumptions used for discounted problems with a restricted control rate (\citep{Alvarez1998, Ekstrom2023, Guerin2024}).

    We also note that 
    since the boundaries are assumed to be natural or entrance-not-exit, item $(iii)$ implies that  
    \begin{equation}\label{Sboundaries}
     S(0+)=-\infty \quad \&\quad S(\infty)=\infty,
    \end{equation} 
    see \citep[p. 234]{Karlin1981}.
    Moreover, (ii) guarantees that 
    \begin{equation*}
        B'(x) = \frac{2 \mu(x)}{\sigma^2(x)} \geq 0
    \end{equation*}
    on $(0,x^*)$.
    Thus, $B$ is nondecreasing on $(0,x^*)$, and in particular the limit $S'(0+):=\lim_{x \downarrow 0}  S'(x) = \exp (- \lim_{x \downarrow 0} B(x))$ exists in $(0, \infty]$. In view of \eqref{Sboundaries}, we then must have $S'(0+)=  \infty$.
    Consequently, the relation \eqref{Eq:ScaleDriftInt} extends so that
\begin{equation}\label{Eq:ScaleDriftInt0}   \int_0^y \mu(z)m(dz) = \frac{1}{S'(y)}
    \end{equation}
    for $y>0$.
    
    Lastly, (i) and (iii) guarantee that a stationary distribution $\pi$ for the uncontrolled diffusion $X$ exists, see \cite[p. 37]{BorodinSalminen}, and is given by
    \begin{equation*}
        \pi(dx) = \frac{m(dx)}{\int_0^{\infty} m(dy)}.
    \end{equation*} 
\end{remark}

\section{Threshold strategies}\label{sec:threshold}

Since the uncontrolled diffusion $X$ is time-homogeneous, it is reasonable to expect that the optimal policy is also time-homogeneous. 
More specifically, we will look for an optimal strategy in the class of {\em barrier rules}, described intuitively in terms of some given threshold $b\geq 0$ as follows. 
\begin{itemize}
    \item[(A)] Assume that the controlled diffusion $X^\nu$ is in the \emph{no-action region} $(0,b)$. Then the decision maker leaves it untouched ($d\nu_t = 0$) and waits.
    \item[(B)] Assume that the controlled diffusion $X^\nu$ is in the \emph{action region} $[b, \infty)$. Then the decision maker applies the control at the maximal possible rate so that $d\nu_t = k(X^\nu_t) dt$ until $X^\nu$ falls below $b$.
\end{itemize}
More precisely, let $b\geq 0$ be a threshold that separates the no-action and action regions. Under our assumptions, the stochastic differential equation 
\begin{equation} \label{Eq:ControlledBarrierDiffusion}
    dX^b_{t} = (\mu(X^b_{t}) - \mathbf{1}_{\{X^b_{t} \geq b\}} k(X^b_{t}) )dt + \sigma(X^b_{t}) dW_{t}, \quad X_0 = x
\end{equation}
admits a unique strong solution on $[0,\tau^b]$, where 
$\tau^b=\inf\{t\geq 0:X^b_t\leq 0\}$,
cf. e.g. \cite[Proposition 5.5.17]{KaratzasShreve1991}. The barrier strategy $\nu=\nu^b$ with threshold $b$ is then defined by
\[    \nu_t=\int_0^{t\wedge\tau^b}\mathbf{1}_{\{X^b_t \geq b\}} k(X^b_t) dt.\]

To study threshold strategies, we note that the controlled process $X^b$ in \eqref{Eq:ControlledBarrierDiffusion} is a diffusion process with drift $\mu(\cdot)-k(\cdot)\mathbf 1_{\{\cdot\geq b\}}$. For a given threshold $b\geq 0$, the scale function $S_b$ of $X^b$ is determined (up to an additive constant) by 
\[    S_b'(x)  = 
        \text{exp}\left(\int_{c_0\vee b}^{x\vee b} \frac{2k(y)}{\sigma^2(y)}dy \right) S'(x),
\]
and the speed measure $m_b$ is given by
\begin{equation*}
        m_b(dx) =
         \text{exp}\left(-\int_{c_0\vee b}^{x\vee b}\frac{2k(y)}{\sigma^2(y)}dy \right)m(dx) .
\end{equation*}
Since $\text{exp}\left(\int_{c_0\vee b}^{x\vee b} \frac{2k(z)}{\sigma^2(z)}dz \right) \geq 1$ for $x\geq c_0$ and $S'(\infty)=\infty$ by (iv) in Assumption~\ref{Ass:MainAssumptions}, we have that $S_0'(\infty)=\infty$.
Similarly, since $\int_0^{\infty} m(dx) < \infty$ by item (iii) of Assumption~\ref{Ass:MainAssumptions}, 
it follows that 
$\int_a^\infty m_b(dx)<\infty$ for any $a>0$.
In particular, for any $b>0$, we have that 
$\int_0^\infty m_b(dy)<\infty$, and 
the diffusion process $X^b$ has a stationary distribution $\pi_b$ given by 
\begin{equation} \label{Eq:StationaryDensityControlled2}
    \pi_b(dx) =  \frac{ \left(\mathbf{1}(x<b) + \text{exp}\left(-\int_b^x \frac{2k(y)}{\sigma^2(y)}dy \right) \mathbf{1}(x \geq b) \right)m(dx)}{\int_0^{b}m(dy) + \int_{b}^{\infty}  \exp\left(-\int_b^y \frac{2k(z)}{\sigma^2(z)}dz \right)m(dy)}.
\end{equation}
On the other hand, if $b=0$, then one applies maximal control all the time, and $X^b=X^0$ does not necessarily have a stationary distribution. 

Below we divide the analysis into two different cases. 
In Section~\ref{b=0} we provide conditions under which 
full harvesting should always be in operation (corresponding to $b=0$).
Then, in Section~\ref{b>0}, we consider the case where a threshold strategy with a strictly positive threshold $b > 0$ is optimal.

The additional assumptions that we impose in each scenario below are formulated in terms of the scale function and the speed measure of the fully controlled process $X^0$ with dynamics 
\begin{equation}\label{X0}
    dX^0_t=(\mu(X^0_t)-k(X^0_t))\,dt + \sigma(X^0_t) dW_t
\end{equation}
for $t\leq \tau^0:=\inf\{t\geq 0:X^0_t\leq 0\}$.
We recall that 
\[S'_0(x) = e^{-B_0(x)} \quad\&\quad m_0(dx) = \frac{2}{\sigma^2(x)} e^{B_0(x)}dx,\]
where 
\[B_0(x) = \int_{c_0}^x \frac{2(\mu(y)-k(y))}{\sigma^2(y)}dy,\]
and that it follows from Assumption~\ref{Ass:MainAssumptions} that $S_0(\infty)=\infty$, $S'_0(\infty)=\infty$ and $\int_a^\infty m_0(dy)<\infty$ for any $a>0$.

\section{A case where one always extracts}
\label{b=0}

In this section we provide conditions under which a threshold strategy with $b=0$ (maximal extraction) is optimal.
In addition to the conditions in Assumption~\ref{Ass:MainAssumptions}, we also impose the following.

\begin{assumption}\label{not hit 0}
    We assume that
    \begin{enumerate}[(i)]
        \item the boundary point $0$ is natural or entrance-not-exit for $X^0$; 
        \item $\int_0^\infty m_0(dy)<\infty $;
        \item $S'_0(0+):=\lim_{x\to 0}S'_0(x)$ exists, with $S'_0(0+)=\infty$;
        \item the quantity \[\overline\mu:=\frac{\int_0^{\infty}\mu(y)m_0(dy)}{\int_0^\infty m_0(dy)}\]
        (which belongs to $[0,\infty)$) satisfies $\overline\mu\leq\mu(0)$.
    \end{enumerate}
\end{assumption}

 \begin{remark}
The assumption (i) that $X^0$ cannot reach 0 in finite time is not very restrictive. Indeed, if instead 0 could be reached, then the long-term payoff would be 0.
Item (ii) guarantees that the fully controlled process is positive recurrent, with $S_0(0)=-\infty$ and with stationary distribution $\frac{m_0(dx)}{\int_0^\infty m_0(dy)}$. 
Finally, (iii) implies that 
 \begin{equation}\label{not hit 00}
 \int_0^y (\mu(z)-k(z)) m_0(dz) = \frac{1}{S'_0(y)},
 \end{equation}
 cf. \eqref{Eq:ScaleDriftInt0}, and therefore 
 \begin{equation}\label{overline mu}
 \overline\mu=\frac{\int_0^{\infty}\mu(y)m_0(dy)}{\int_0^\infty m_0(dy)}=\frac{\int_0^{\infty}k(y)m_0(dy)}{\int_0^\infty m_0(dy)}\in [0,\infty).
 \end{equation}
\end{remark}

Under Assumption~\ref{not hit 0}, the ergodic average of the strategy $\nu^0$ where one always extracts is given by 
\begin{equation}\label{beta*}
 J(x;\nu^0)=\liminf_{T\to \infty}\frac{1}{T}\mathbb{E}_x\left[\int_0^T k(X^0_t) dt \right] = \frac{\int_0^\infty k(y) m_0(dy)}{\int_0^\infty m_0(dy)}=:\beta, 
\end{equation}
see \cite[p. 37]{BorodinSalminen}.
Consequently, 
\[\sup_{\nu \in \mathcal{V}} J(x;\nu) \geq \beta.
\]
In this section, we show that 
\[\sup_{\nu \in \mathcal{V}} J(x;\nu) = \beta,\]
and, in particular, that $\nu^0$ is an optimal strategy.

To this end, we use an auxiliary function $u \in \mathcal{C}^2$ that solves an appropriate ordinary differential equation. More specifically, let $u$ solve the equation
\begin{equation}\label{HJB-always}
\mathcal{A}u(x) - k(x)u'(x) + k(x) = \beta, \quad x \in (0,\infty),
\end{equation}
where 
\begin{align*} 
    \mathcal{A} &= \dfrac{\sigma^2(x)}{2} \dfrac{d^2}{dx^2} + \mu(x)\dfrac{d}{dx} 
\end{align*}
is the infinitesimal generator of the process $X$.
Additionally, we impose the boundary condition $\frac{u'(0+)}{S_0'(0+)}=0$ so that
\begin{equation}\label{u}
u'(x) = S_0'(x)\int_0^x (\beta-k(y))m_0(dy).
\end{equation}

\begin{lemma}\label{Lip}
Let Assumption~\ref{not hit 0} hold.
Then the function $u$ defined (up to an additive constant) in \eqref{u} satisfies
    $0\leq u'(x) \leq 1$ for all $x \in(0,\infty)$.
\end{lemma}

\begin{proof}
    First note that  
    \begin{align*}
    \frac{u'(x)}{S_0'(x)} &= \int_0^x (\beta-k(y))m_0(dy) \\
    & = \int_0^xm_0(dy)\left(\frac{\int_0^{\infty}k(y)m_0(dy)}{\int_0^\infty m_0(dy)} - \frac{\int_0^x k(y) m_0(dy)}{\int_0^x m_0(dy)} \right).
    \end{align*}
    Since $k$ is non-decreasing, the function 
    \[x\mapsto \frac{\int_0^x k(y) m_0(dy)}{\int_0^x m_0(dy)}\]
    is also non-decreasing, so
    $u'(x) \geq 0$.

    Next, define
    \begin{equation*}
        F(x) = \int_0^x (\mu(y)-\overline{\mu}) m_0(dy).
    \end{equation*} 
    As by our assumptions $\mu$ is unimodal with $\mu(0) \geq \overline{\mu} > \mu(\infty)$, there is a unique solution $x_0>0$ of  $\mu(x_0) = \overline{\mu}$. Moreover, $F'(x) =  (\mu(x)-\overline{\mu})m_0'(x)$ is positive on $(0,x_0)$ and negative on $(x_0, \infty)$. Since $F(0) = 0$ and $F(\infty) = 0$, we find that
    \begin{equation*}
        F(x) \geq 0, \quad \text{ for all } \, x >0.
    \end{equation*}
    Therefore, using \eqref{not hit 00}, we have
    \begin{align*}
    \frac{u'(x)-1}{S_0'(x)} &= \int_0^x (\beta-k(y))m_0(dy) - \frac{1}{S_0'(x)} \\
    & = \int^x_0 (\beta-k(y))m_0(dy) - \int_0^x(\mu(y)-k(y)) m_0(dy)\\
    & = \beta \int_0^x m_0(dy)-\int_0^x \mu(y) m_0(dy)  
     = -F(x) \leq 0,
    \end{align*}
    proving that $u'(x)\leq 1$.
\end{proof}

\begin{theorem}\label{verification-always}
Let Assumption~\ref{not hit 0} hold. 
    Then we have 
    \[\sup_\nu J(x;\nu) =\beta.\]
Consequently, in view of \eqref{beta*}, $\nu^0$ is an optimal strategy, i.e. $J(x;\nu^0)=\beta$.
\end{theorem}

\begin{proof}
     Let $T > 0$, let $\nu$ be an admissible strategy with  corresponding controlled process $X^\nu$, and define 
     $T_n=T\wedge\tau_n$, where $\tau_n := \inf \{ t\geq 0: X^\nu_t \notin (1/n,n) \}$. Then, by an application of It\^o's formula, 
     we get
    \[
    u(X^\nu_{T_n}) - u(x) = \int_0^{T_n} \mathcal{A}u(X^\nu_s) ds - \int_0^{T_n} u'(X^\nu_s) d\nu_s + \int_0^{T_n} u'(X^\nu_s)\sigma(X^\nu_s) dW_s.
    \]
    Taking expectations yields
    \[
        \mathbb{E}_x [u(X^\nu_{T_n}) ] - u(x) = \mathbb{E}_x \left[ \int_0^{T_n} \mathcal{A}u(X^\nu_s) ds - \int_0^{T_n} u'(X^\nu_s) d\nu_s \right],
    \]
    and, since $u$ solves the differential equation \eqref{HJB-always}, we find that
    \begin{align*}
        & \mathbb{E}_x [u(X^\nu_{T_n})] - u(x) 
         = \mathbb{E}_x\left[\int_0^{T_n} (\beta - k(X^\nu_s))ds \right] - \mathbb{E}_x\left[\int_0^{T_n}u'(X^\nu_s)( d\nu_s - k(X^\nu_s)ds ) \right].
    \end{align*}
    Using $d\nu_t \leq k(X^\nu_t) dt$ and the bound $u'(x) \leq 1$ (see Lemma~\ref{Lip}), we find that
    \[ \mathbb{E}_x [u(X^\nu_{T_n})] - u(x) 
         \leq \mathbb{E}_x \left[ \int_0^{T_n} (\beta - k(X^\nu_s))ds \right] - \mathbb{E}_x \left[\int_0^{T_n}( d\nu_s - k(X^\nu_s)ds ) \right],
    \]
    so
    \[
    \mathbb{E}_x [u(X^\nu_{T_n})] - u(x) + \mathbb{E}_x \left[\int_0^{T_n} d\nu_s \right] \leq \mathbb{E}_x \left[\int_0^{T_n} \beta ds \right].
    \]
    Taking the limit as $n\to\infty$, we have $T_n\to T$ a.s. Using that $u$ is bounded from below (by Lemma~\ref{Lip}), we get by monotone convergence and Fatou's lemma that
    \[    \mathbb{E}_x [u(X^\nu_{T}) ] - u(x) + \mathbb{E}_x \left[\int_0^{T} d\nu_s \right] \leq \beta T.
    \]
    Also, since $u$ is bounded from below, it follows that
    \[J(x;\nu)=\liminf_{T\to\infty}\frac{1}{T}\mathbb E_x[\nu_T]\leq \beta,\]    
    which finishes the proof.
\end{proof}

\section{A case with a positive threshold}
\label{b>0}

In this section we provide conditions under which a threshold strategy with a strictly positive threshold is optimal. 
In addition to Assumption~\ref{Ass:MainAssumptions}, we impose the following conditions.

\begin{assumption}\label{hit 0}
We have that
\begin{enumerate}[(i)]
    \item 
    $\int_0^\infty m_0(dy)<\infty$;
    \item 
    the quantity 
\[\overline\mu:=\frac{\int_0^{\infty}\mu(y)m_0(dy)}{\int_0^\infty m_0(dy)}\]
    satisfies $\mu(0)<\overline\mu$.
\end{enumerate}
\end{assumption}

Recall from Section~\ref{sec:threshold} the notation $X^b$ (the diffusion obtained from a threshold strategy with threshold $b$) and $\pi_b$ (the stationary distribution of $X^b$ when $b>0$).
If a threshold strategy with $b>0$ is applied, then the average long-term payoff is given by
\begin{eqnarray*}
J(x;\nu^b) &=&    \liminf_{T\to\infty}\frac{1}{T}\mathbb E_x\left[\int_0^Tk(X_t^b) \mathbf{1}_{\{X^b_t \geq b\}} dt\right] =
    \int_0^{\infty}  k(y) \mathbf{1}(y \geq b) \pi_b(dy)\\
   &=& \frac{\int_b^{\infty} k(y)\exp\left(-\int_b^y \frac{2k(z)}{\sigma^2(z)}dz \right)m(dy)}{\int_0^{b}m(dy) + \int_{b}^{\infty} \exp\left(-\int_b^y \frac{2k(z)}{\sigma^2(z)}dz \right)m(dy)}.
    \end{eqnarray*}
Straightforward calculations, using (\ref{Eq:ScaleDriftInt}), then show that
\begin{equation}\label{beta(b)}
    J(x;\nu^b) = \frac{\exp\left(-\int_c^b\frac{2k(z)}{\sigma^2(z)}dz\right)\int_0^{b} \mu(y)m(dy) +
\int_{b}^{\infty} \mu(y)m_0(dy)}{\exp\left(-\int_c^b\frac{2k(z)}{\sigma^2(z)}dz\right)\int_0^{b}m(dy)+\int_{b}^{\infty} m_0(dy)}=:\beta(b).
\end{equation}
    
 We note by standard differentiation that if $b>0$ maximizes the expression in \eqref{beta(b)}, then $b$ must satisfy the first order condition
 \begin{align*}
&\left(\exp\left(-\int_c^b\frac{2k(z)}{\sigma^2(z)}dz\right)\int_0^{b}m(dy) + \int_{b}^{\infty} m_0(dy)\right)\int_0^{b}m(dy)\\
& =
 \left(\exp\left(-\int_c^b\frac{2k(z)}{\sigma^2(z)}dz\right)\int_0^{b}\mu(y)m(dy) + \int_{b}^{\infty} \mu(y)m_0(dy)\right)\int_0^b m(dy).
 \end{align*}
Simplifications, using formula (\ref{Eq:ScaleDriftInt}), lead to the equation
\begin{align} \label{Eq:b}
      \frac{\int_0^{b} \mu(y)m(dy)}{\int_0^{b}m(dy)}=  \frac{\int_{b}^{\infty} \mu(y)m_0(dy)}{\int_{b}^{\infty} m_0(dy)}.
\end{align}
Using this condition in the expression \eqref{beta(b)} for $\beta(b)$, we note that if $b$ satisfies \eqref{Eq:b}, then 
\begin{equation*}
    \beta(b) = \frac{\int_0^{b} \mu(y)m(dy)}{\int_0^{b}m(dy)} =\frac{\int_{b}^{\infty} \mu(y)m_0(dy)}{\int_{b}^{\infty} m_0(dy)}.
\end{equation*}

To investigate the existence and uniqueness of solutions to \eqref{Eq:b}, we introduce the functions
\begin{align*}
    G(x) = \frac{\int_0^{x} \mu(y)m(dy)}{\int_0^{x}m(dy)}, \quad
    H(x) =\frac{\int_{x}^{\infty} \mu(y)m_0(dy)}{\int_{x}^{\infty} m_0(dy)},
\end{align*}
and note that equation \eqref{Eq:b} then reads $G(b)=H(b)$.
Denote $x_0>x^*$ the unique positive solution to $\mu(x_0)=\mu(0)$.

\begin{lemma} \label{Lemma: Gshape}
    There exists a point $x'\in(0,x_0)$ such that the function $G$ 
    is increasing and satisfies $G<\mu$ on $(0,x')$, and decreasing with $G>\mu$ on $(x',\infty)$.    
\end{lemma}

\begin{proof}
    Direct differentiation yields that
    \begin{align*}
        G'(x) = \frac{m'(x)}{\left( \int_0^x m(dy) \right)^2} \left( \mu(x) \int_0^x m(dy)- \int_0^x \mu(y)m(dy) \right) 
        = \frac{m'(x)}{ \int_0^x m(dy)}(\mu(x)-G(x)). 
    \end{align*}
    Moreover, $\mu(x^*)-G(x^*) > 0$ and $\mu(x_0)-G(x_0) < 0$, which shows by continuity that there exists a root $x'$ to $G(x') = \mu(x')$ with $x' \in (0, x_0)$. Furthermore, it is clear from the expression for $G'$ that $G$ is increasing on $(0,x')$ and decreasing on $(x',\infty)$.
\end{proof}

\begin{remark}
The point $x'$ is the optimal threshold in the corresponding singular ergodic control problem (cf. \cite{AlvarezHening2022} and \cite{LiangLiuZervos2025}).    
\end{remark}

\begin{lemma} \label{Lemma: Uniqueb}
    Equation \eqref{Eq:b} admits a unique solution $b^*\in(0,\infty)$. Moreover, $b^*\in(0,x')$.
\end{lemma}

\begin{proof}
    It is clear that 
    \[G(0+)=\mu(0)<\overline\mu=H(0+)\] 
    and 
    \[G(x)>\mu(x)> H(x)\]
    for $x>x'$. This, together with continuity of $G$ and $H$, shows 
    the existence of a solution on $(0,x')$, and that no solution exists on $[x',\infty)$.
    
    Also, if $\overline x\in(0,x')$ solves $G(\overline x)=H(\overline x)$, then $G'(\overline x)\geq 0$ and
    \begin{eqnarray*}
        H'(\overline x) = m_0'(\overline x)\frac{H(\overline x)-\mu(\overline x)}{\int_{\overline x}^\infty m_0(dy)}
        = m_0'(\overline x)\frac{G(\overline x)-\mu(\overline x)}{\int_{\overline x}^\infty m_0(dy)} <0.
    \end{eqnarray*}
    From this, uniqueness of solutions to Equation~\eqref{Eq:b} follows.
\end{proof}

The following lemma further restricts the interval for the candidate optimal threshold $b^*$.

\begin{lemma} \label{Lemma: Hshape}
    There exists a point $x_H \in(0,b^*)$ such that the function $H$ is increasing and satisfies $H> \mu$ on $(0,x_H)$, and decreasing with $H<\mu$ on $(x_H,\infty)$. Consequently, $b^* \in (x_H, x')$.
\end{lemma}
\begin{proof}
Differentiation shows that the sign of the derivative 
\begin{eqnarray} \label{Eq: H'sign}
    H'(x) = \frac{m_0'(x)\left(\int_{x}^{\infty} \mu(y)m_0(dy) -\mu(x)\int_{x}^{\infty} m_0(dy)\right)}{\left(\int_{x}^{\infty} m_0(dy)\right)^2}
    = \frac{m_0'(x)\left(H(x) -\mu(x)\right)}{\int_{x}^{\infty} m_0(dy)}
\end{eqnarray}
coincides with the sign of  $H(x)-\mu(x)$.
Also, we have $H(0)-\mu(0) = \bar{\mu}-\mu(0)$, which is strictly positive due to Assumption \ref{hit 0}. Moreover, from the proof of Lemma \ref{Lemma: Uniqueb} we have that $H(b^*) - \mu(b^*) < 0$. Hence by continuity it follows that there exists $x_H \in (0,b^*)$ such that $H(x_H) = \mu(x_H)$.

Finally, we prove that $x_H$ is unique. We will proceed by contradiction. Assume that there are two roots of $H(x) - \mu(x)$, denoted $x_H^1$ and $x_H^2$, and with $x_H^1 < x_H^2$. Note that these roots lie in $(0,x^*)$, since for $x \geq x^*$ the monotonicity of $\mu$ gives $H(x)<\mu(x)$. Also, by continuity, we may assume that either $H-\mu\leq 0$ or $H-\mu>0$ in $(x_H^1,x_H^2)$.

Suppose first that $H(x)-\mu(x) \leq 0$ in $(x_H^1, x_H^2)$. Note that $H$ is decreasing on $(x_H^1, x_H^2)$ by \eqref{Eq: H'sign}. Since $\mu$ is increasing on $(0,x^*)$, we get that $H(x_H^2) = \mu(x_H^2) > \mu(x_H^1) = H(x_H^1),$ which is a contradiction. 

On the other hand, suppose that $H(x)-\mu(x) > 0$ in $(x_H^1, x_H^2)$, and fix $x \in (x_H^1, x_H^2)$. Multiplying the condition $H(x_H^1)=\mu(x_H^1)$ by $\int_{x_H^1}^\infty m_0(dy)$ and using that $\mu$ is increasing on $(0,x^*)$, we obtain
\begin{align*}
    0 & = \int_{x_H^1}^{x} \left(\mu(y)-\mu(x_H^1)\right) m_0(dy)
     + \int_{x}^{\infty} \left(\mu(y)-\mu(x_H^1)\right) m_0(dy) \\
    & \geq \int_{x}^{\infty} \left(\mu(y)-\mu(x)\right) m_0(dy)
     = \left(H(x)-\mu(x)\right)\int_{x}^{\infty} m_0(dy) > 0,
\end{align*}
which is again a contradiction. Hence the root $x_H$ is unique, so that $H > \mu$ on $(0, x_H)$ and $H < \mu$ on $(x_H, \infty)$, and by \eqref{Eq: H'sign} the function $H$ is increasing on $(0, x_H)$ and decreasing on $(x_H,\infty)$. Finally, since $H(b^*)<\mu(b^*)$, we have $b^* > x_H$, and thus $b^* \in (x_H, x')$ by Lemma~\ref{Lemma: Uniqueb}.

\end{proof}

In Figure \ref{Fig: PrototypeGH}, illustrative versions of the functions $\mu$, $G$, and $H$ are plotted with their maxima and intersection points $x^*$, $x_H$, $b^*$ and $x'$. Note that the orderings $x_H<b^*<x'$ and $x_H<x^*<x'$ are known to hold; Figure~\ref{Fig: PrototypeGH} depicts a situation in which, additionally, $b^*<x^*$.

\begin{figure} 
    \centering
    \includegraphics[width=0.85\linewidth]{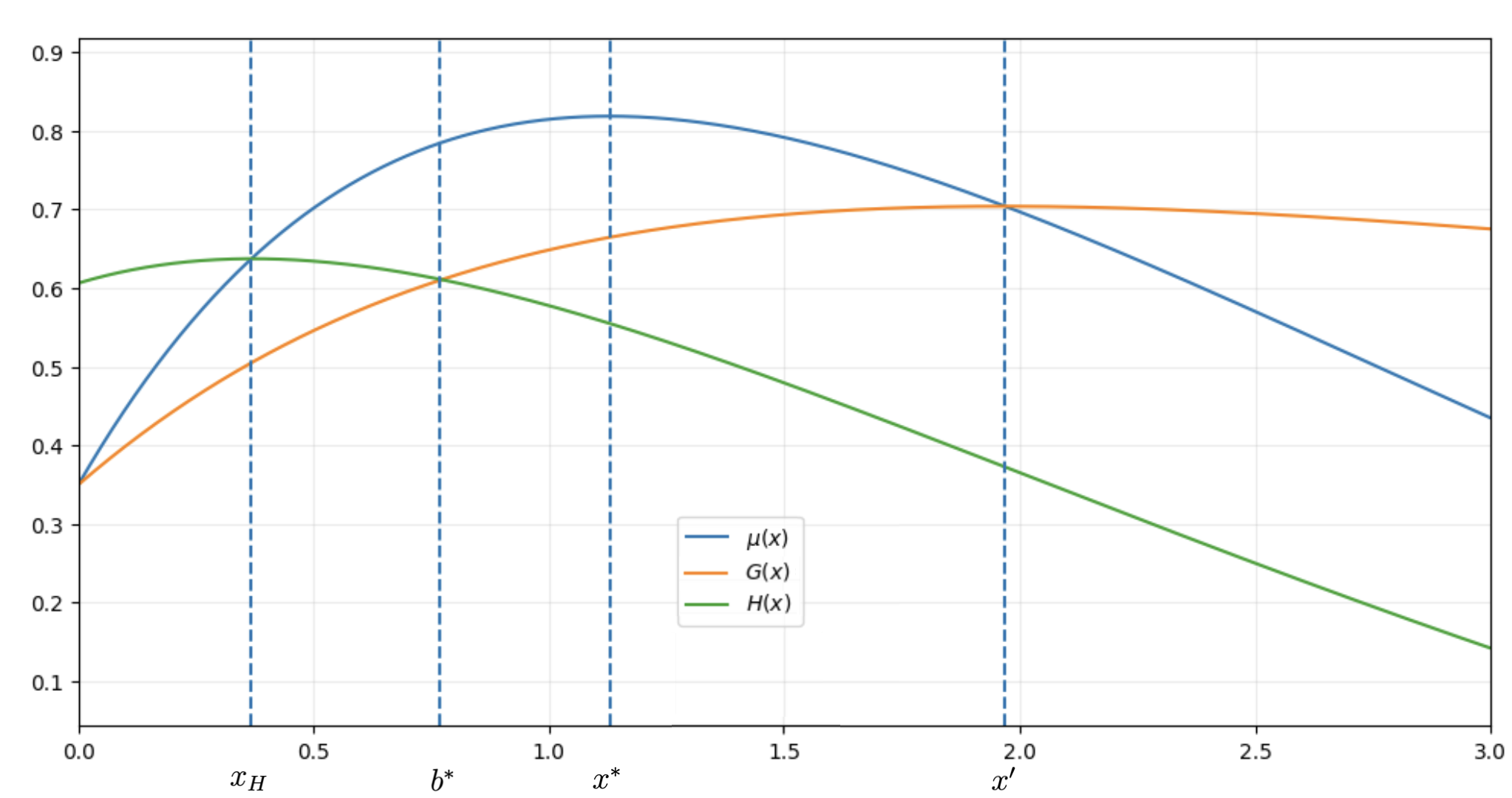}
    \caption{Schematic image of the functions $\mu$, $G$, and $H$ together with their maximum and intersection points $x^*$, $x_H$, $b^*$ and $x'$.}
    \label{Fig: PrototypeGH}
\end{figure}

\subsection{Verification theorem}

Since the proposed barrier policy $\nu=\nu^{b^*}$ is an admissible strategy, we necessarily have that 
\begin{equation}\label{ineq}
    \sup_{\nu \in \mathcal V} J(x;\nu)\geq \beta(b^*)=:\beta^*.
\end{equation}
Our next goal is to establish the opposite inequality, so that 
\[\sup_{\nu \in \mathcal V} J(x;\nu)=\beta^*\] 
and so that $\nu^{b^*}$ is an optimal strategy.

To do that, one may try to copy the arguments from Section~\ref{b=0} and construct an auxiliary function 
$u\in C^2((0,\infty))$ such that 
\begin{align*}
     &\mathcal{A}u(x) = \beta^*, & u'(x)&\geq 1, & x < b^*, \\
     &\mathcal{A}_0 u(x) = \beta^*-k(x),& u'(x)&\leq 1,  & x \geq b^*, 
\end{align*}
with $\frac{u'(0+)}{S'(0+)} = \frac{u'(\infty)}{S'_0(\infty)} =0$.
In fact, it is straightforward to check that $u$ is then given (up to an additive constant) by 
\[u'(x)=\left\{\begin{array}{ll}
1-S'(x)(G(x)-\beta^*)\int_0^x m(dy) & x<b^*\\
1+S_0'(x) (H(x)-\beta^*) \int_x^{\infty} m_0(dy) & x \geq b^*.
\end{array}\right.\]
However, the function $u$ is not necessarily lower bounded (cf. \citep{Cohen2022}), which makes it hard to use directly in a verification argument.
Instead, the next proposition defines a family of auxiliary functions $u_\beta$, where each function $u_\beta$ is lower bounded and satisfies an approximate HJB 
equation. This construction is an extension of the argument in \citep{LiangLiuZervos2025}, accommodating for the current setting of a restricted control rate.

For this, let $\overline\beta=\sup_bH(b)$, and
define the function $\delta: [\beta^*, \overline{\beta}] \to [x_H, b^*]$ as the unique solution of 
$H(\delta(\beta)) = \beta$ with $\delta(\overline{\beta}) \geq  x_H$, and note that $\delta(\beta)\leq b^*$ and 
$\delta(\beta^*) = b^*$ (see Lemma \ref{Lemma: Hshape} and Figure \ref{Fig: PrototypeGH} to see that $\delta$ is well-defined).

\begin{proposition}\label{Prop:Approx:Solution}
    Given any $\beta \in (\beta^*, \overline{\beta})$, there exists a unique point $\alpha(\beta) \in (0, \delta(\beta))$, such that the unique (up to a constant) continuous function $u_\beta$ defined by
    \begin{equation*}
        u'_{\beta}(x) = 
        \begin{cases}
            1, & x \in (0, \alpha(\beta)], \\
            1+S'(x)\int^{\delta(\beta)}_{x} (\mu(y) - \beta)m(dy), & x \in (\alpha(\beta), \delta(\beta)), \\
            1+S_0'(x) \int_x^{\infty} (\mu(y)-\beta) m_0(dy), & x \in [\delta(\beta), \infty),
        \end{cases}
    \end{equation*}
    is in $\mathcal{C}^1((0,\infty))\cap \mathcal{C}^2((0,\infty)\setminus\{\alpha(\beta),\delta(\beta)\})$.

    Furthermore, $u_{\beta}$ satisfies
  \begin{align} \label{Eq: u'betalower}
       \mathcal A u_\beta(x)&\leq \beta,  &  u'_{\beta}(x) &\geq 1,   & x \in (0, \delta(\beta)), \\
       \mathcal A_0u_\beta(x) &=\beta-k(x), &  u'_{\beta}(x) &\leq 1,  & x \in [\delta(\beta), \infty).  
      \label{Eq: u'betaupper}
    \end{align}
\end{proposition}

\begin{proof}
    Let $\beta \in (\beta^*, \overline{\beta})$. 
    We first note that $u_{\beta}$ satisfies $\mathcal{A}u_{\beta} = \beta$ on $(\alpha(\beta), \delta(\beta))$ and $\mathcal{A}_0u_{\beta} = \beta-k$ on $(\delta(\beta), \infty)$. Moreover, for $x\geq\delta(\beta)$ we have 
    \[u_{\beta}'(x) = 1+S_0'(x)(H(x)-\beta)\int_x^\infty m_0(dy),\]
    so $u_{\beta}'(\delta(\beta)) = 1$.
    Hence, the continuous differentiability only requires us to find the function $\alpha$ such that
    \begin{equation} \label{Eq: alphadef}
        \int_{\alpha(\beta)}^{\delta(\beta)} (\mu(y)-\beta) m(dy) = 0.
    \end{equation}
    Consider the function
    \begin{equation*}
       f(x) :=  \int_{x}^{\delta(\beta)} (\mu(y)-\beta) m(dy).
    \end{equation*}
    Then $f(\delta(\beta)) = 0$, and differentiation shows that $f'(x) =-(\mu(x)-\beta)m'(x)$. Since $\delta(\beta) > x_H$, there exists a unique turning point $x_f < x_H$ so that $f$ is increasing on $(0,x_f)$ and then decreasing on $(x_f,\delta(\beta))$. Furthermore,
    \begin{equation*}
        f(0) =  \int_{0}^{\delta(\beta)} (\mu(y)-\beta) m(dy) = (G(\delta(\beta))-\beta) \int_0^{\delta(\beta)}m(dy) < 0.
    \end{equation*}
    Thus by continuity we get the existence and uniqueness of the required $\alpha(\beta)$. 

    Moreover, note that $\mu(\alpha(\beta))\leq \beta$, so 
    \[\mathcal A u_\beta(x)=\mu(x)\leq\mu(\alpha(\beta))\leq \beta\]
    for $x\in(0,\alpha(\beta))$.

    It remains to establish the second inequalities in \eqref{Eq: u'betalower} and \eqref{Eq: u'betaupper}. Clearly, again by definition of $\delta$,
    \begin{equation*}
        u_{\beta}'(x)- 1  = S_0'(x) (H(x)-\beta) \int_x^{\infty} m_0(dy) \leq 0, \quad \text{for all } x \in [\delta(\beta),\infty).
    \end{equation*}
    Similarly,  
    \begin{equation*}
        u_{\beta}'(x)-1  = S'(x)\int^{\delta(\beta)}_{x} (\mu(y)-\beta)m(dy)= S'(x)f(x) \geq 0 
    \end{equation*}
    for $x \in (\alpha(\beta),\delta(\beta)]$.
\end{proof}

We are now ready to show that the inequality in \eqref{ineq} is, in fact, an equality.

\begin{theorem}[Verification]
Let Assumptions \ref{Ass:MainAssumptions} and \ref{hit 0} hold. Suppose that the candidate pair $(\beta^*,b^*)$ is the unique solution to
\begin{equation*}
    \beta^*=G(b^*)=H(b^*), \quad b^* \in (x_H,x').
\end{equation*}
Then 
\begin{equation*}
  \sup_{\nu \in \mathcal V}  \liminf_{T \to \infty}\frac{1}{T} \mathbb{E}_x \left[ \int_0^T d \nu_t \right] = \beta^*.
\end{equation*}
\end{theorem}

\begin{proof}
Fix $\beta \in (\beta^*, \overline{\beta})$ and $\nu \in \mathcal V$. For $\rho > 0$ define the almost surely finite stopping times
\begin{equation*}
    \tau_\rho:=\inf\{t\ge 0: X_t^\nu \geq \rho\}, \quad \tau:= T \wedge \tau_\rho.
\end{equation*}
Since $u_\beta \in C^1(0,\infty)$ and $u_\beta$ is $C^2$ on $(0,\alpha(\beta))$, $(\alpha(\beta),\delta(\beta))$ and $(\delta(\beta),\infty)$, a generalized It\^o formula yields
\begin{align*}
    u_\beta(X^\nu_{\tau})
    &=u_\beta(x)
    +\int_0^{\tau}\mathcal A u_\beta(X_s^\nu) ds
    +\int_0^{\tau}\sigma(X_s^\nu)u'_\beta(X_s^\nu) dW_s
    -\int_0^{\tau}u'_\beta(X_s^\nu) d\nu_s .
\end{align*}
Taking expectations gives
\begin{equation} \label{Eq:Itoub}
    \mathbb{E}_x [u_\beta(X^\nu_{\tau})]
    =u_\beta(x)+\mathbb{E}_x \left[\int_0^{\tau} \mathcal A u_\beta(X_s^\nu) ds \right]
    -\mathbb{E}_x \left[\int_0^{\tau} u'_\beta(X_s^\nu) d\nu_s \right].
\end{equation}
By construction of $u_\beta$ we have
\begin{equation*}
    \mathcal A u_\beta(x) = 
    \begin{cases}
        \mu(x), & x\in(0,\alpha(\beta)), \\
        \beta, & x\in(\alpha(\beta),\delta(\beta)), \\
        \beta+(u'_\beta(x)-1)k(x), &  x \in [\delta(\beta), \infty).
    \end{cases}
\end{equation*}
Since $\mu(x) \leq \beta$ on $(0, \alpha(\beta))$, we find using the above that
\begin{align*}
    \mathbb{E}_x[ u_\beta(X^\nu_{\tau}) ]
    & \leq u_\beta(x) - \mathbb{E}_x \left[\int_0^{\tau}u'_\beta(X_s^\nu) d\nu_s \right] \\
    & + \beta \mathbb{E}_x [\tau] 
    + \mathbb{E}_x \left[ \int_0^{\tau}  (u'_{\beta}(X_s^\nu)-1)k(X_s^\nu) \mathbf{1}_{[\delta(\beta),\infty)}(X_s^\nu) ds \right].
\end{align*}
Using the inequality $u_{\beta}'(x) \leq 1$ on $[\delta(\beta), \infty)$ and the bound $d\nu_s\leq k(X_s^\nu) ds$, we observe that
\begin{align*}
    \mathbb{E}_x[ u_\beta(X^\nu_{\tau}) ]
    & \leq u_\beta(x) - \mathbb{E}_x \left[\int_0^{\tau} (u'_\beta(X_s^\nu) - (u'_{\beta}(X_s^\nu)-1)
    \mathbf{1}_{[\delta(\beta),\infty)}(X_s^\nu)
   )d\nu_s \right]
    + \beta \mathbb{E}_x [\tau]. 
\end{align*}
Note that $\mathbf{1}_{[\delta(\beta),\infty)}(x)(1-u'(x)) + u'(x) \geq 1$, which gives
\begin{align*}
    \mathbb{E}_x[ u_\beta(X^\nu_{\tau}) ]
    \leq u_\beta(x)+ \beta \mathbb{E}_x [\tau] - \bigg[\int_0^{\tau}  d \nu_s \bigg]. 
\end{align*}
Divide by $T$ (recall $\tau=T\wedge\tau_\rho\le T$), and let first $\rho\to\infty$ and then $T\to\infty$ to obtain
\begin{align*}  
    \liminf_{T \to \infty}\frac{1}{T}\mathbb{E}_x[u_{\beta}(X_{T}^{\nu})] +  \liminf_{T \to \infty}\frac{1}{T}\mathbb{E}_x \bigg[\int_0^{T}  d \nu_s \bigg] & \leq  \beta . 
\end{align*}
Since $u_{\beta}$ is bounded from below, we find that
\begin{equation*}
    \liminf_{T \to \infty} \frac{1}{T} \mathbb{E}_x \left[\int_0^T d\nu_s \right] \leq \beta.
\end{equation*}
Finally, since this holds for any $\beta \in (\beta^*,\overline\beta)$,
the result follows.
\end{proof}

\section{Comparative analysis}\label{Sec:Comparative}

It is natural to expect that a more lenient harvesting bound should allow the controller to keep the population untouched at a higher threshold. Indeed, when harvesting can be carried out more aggressively, the agent can afford to wait until the stock reaches a larger level before intervening, because any excess can then be removed more rapidly. In this sense, a higher harvesting bound should move the optimal threshold upward, and ultimately toward the singular threshold $x'$.

The next proposition confirms this intuition by showing that if the bound on the harvesting rate is increased, then the corresponding optimal threshold becomes larger.

\begin{proposition} \label{Prop: k1k2ordering}
    Let $k_1,k_2: (0,\infty) \to (0,\infty)$ be two extraction bounds satisfying Assumptions~\ref{Ass:MainAssumptions} and \ref{hit 0}, with $k_1(x) \leq k_2(x)$ for all $x \geq 0.$
    
    For $i=1,2$, let
    \begin{equation*}
        H_i(x) := \frac{\int_x^\infty \mu(y)m_{i}(dy)}{\int_x^\infty m_{i}(dy)},
    \end{equation*}
    where $m_{i}$ is the speed measure corresponding to the drift $\mu-k_i$, and let $b_i$ denote the unique positive solution of $G(b_i) = H_i(b_i).$ Then $b_1 \leq b_2.$
\end{proposition}
\begin{proof}
    Let 
    \begin{equation*}
       \lambda_i(x) := \frac{m_{i}'(x)}{\int_x^\infty m_{i}(dy)}\quad\& \quad  \Delta(x) := H_2(x)-H_1(x).
    \end{equation*}
By the definition of the speed measure we have $m_{2}'(y) = q(y)m_{1}'(y),$ where
    \begin{equation*}
        q(y) = \exp{\left(-\int^{y}_{c_0}\frac{2(k_2(z)-k_1(z))}{\sigma^2(z)}dz \right)}.
    \end{equation*}
    As $k_2(z)-k_1(z) \geq 0$ for all $z>0$, the function $q$ is nonincreasing. Therefore we find
    \begin{equation}\label{lambdaineq}
        \lambda_2(x) = \frac{q(x)m_{1}'(x)}{\int_x^\infty q(y)m_{1}'(y)dy} \geq \frac{m_{1}'(x)}{\int_x^\infty m_{1}'(y)dy} = \lambda_1(x).
    \end{equation}
    
    Next, differentiation yields $H_i'(x) =  \lambda_i(x) \big(H_i(x)-\mu(x)\big)$, and thus
    \begin{equation*}
        \Delta'(x)-\lambda_2(x)\Delta(x) = \big( \lambda_2(x)-\lambda_1(x)\big)\big(H_1(x)-\mu(x) \big).
    \end{equation*}
    Let $x_H^1$ be the point at which $H_1-\mu$ changes sign, cf. Lemma~\ref{Lemma: Hshape}. Then
    \begin{equation*}
        H_1(x)-\mu(x) \leq 0 \quad \text{ for all } x \geq x_H^1.
    \end{equation*}
    Combining this with \eqref{lambdaineq}, we obtain for all $x \geq x_H^1$ that 
    \begin{equation} \label{Eq: DeltaIneq}
        \Delta'(x)-\lambda_2(x)\Delta(x) \leq 0.
    \end{equation}
    Consequently,
    \begin{equation*}
        \frac{d}{dx} \left( \Delta(x)\int_x^\infty m_{2}(dy) \right) = \int_x^\infty m_{2}(dy) \big(\Delta'(x)-\lambda_2(x)\Delta(x)\big) \leq 0,
    \end{equation*}
    and thus the function $x \mapsto \Delta(x) \int_x^\infty m_{2}(dy)$ is nonincreasing on $[x_H^1,\infty)$. 
    
    Now, since $q$ is nonincreasing, the ratio 
    \begin{equation*}
        \frac{\int_x^\infty m_{2}(dy)}{\int_x^\infty m_{1}(dy)}
        =\frac{\int_x^\infty q(y)m_{1}(dy)}{\int_x^\infty m_{1}(dy)}
    \end{equation*}
    is bounded by $q(x_H^1)$ on $[x_H^1, \infty)$. Noting that 
    \[\int_x^\infty \mu(y)m_{i}(dy) \to 0\] 
    as $x \to \infty$, it follows that
    \begin{equation*}
        \Delta(x)\int_x^\infty m_{2}(dy) = 
        \int_x^\infty \mu(y)m_{2}(dy)-\int_x^\infty \mu(y)m_{1}(dy)(x)\frac{\int_x^\infty m_{2}(dy)}{\int_x^\infty m_{1}(dy)} \to 0
    \end{equation*}
    as $x \to \infty$. Hence, we have 
    \begin{equation*}
        \Delta(x) \int_x^\infty m_{2}(dy) \geq 0 \quad \text{ for all } x \geq x_H^1,
    \end{equation*}
    which implies
    \begin{equation*}
        \Delta(x) \geq 0 \quad \text{ for all } x \geq x_H^1.
    \end{equation*}
    Consequently, $H_2(x) \geq H_1(x)$ for all $ x \geq x_H^1$.
    
    Finally, since the unique root $b_1 \in (x_H^1,x')$, we have 
    \begin{equation*}
        H_2(b_1) \geq H_1(b_1) = G(b_1).
    \end{equation*}
    As the equation $G(x) = H_2(x)$ has the unique solution $b_2$, and since $G(x)-H_2(x)$ changes sign from negative to positive at that point, it follows that $b_1 \le b_2.$
\end{proof}

If the maximal harvesting rate becomes very large, then harvesting is no longer effectively constrained by the rate bound. Biologically, this means that whenever the population exceeds the desired level, the controller can remove biomass almost instantaneously. One should therefore expect the bounded-rate harvesting problem to approach the singular harvesting problem studied in \cite{Alvarez2018}, in which the population is reflected at the optimal threshold level by instantaneous extraction.

The next proposition shows that, as the harvesting bound is increased to infinity, the optimal threshold $b_M^*$ converges to the singular optimal threshold $x'$. At the same time, the optimal long-run average payoff converges to $\mu(x')$, which is exactly the asymptotic yield associated with maintaining the population at that level by the singular reflection policy.

\begin{proposition} \label{Prop: limitharvesting}
    Assume that the extraction bound is for $M\geq 1$ replaced by $k_M(x):=M k(x)$, where $k$ satisfies Assumptions~\ref{Ass:MainAssumptions} and \ref{hit 0}, and let $b_M^*$ and $\beta_M^*$ denote the corresponding optimal threshold and optimal average harvesting rate, respectively. Then
    \begin{equation*}
        \lim_{M \to \infty} b_M^* = x' \quad\&\quad \lim_{M \to \infty} \beta_M^* = \mu(x').
    \end{equation*}
\end{proposition}

\begin{proof} 
    Since $k$ is non-decreasing and strictly positive, so is $M k$ for $M \geq 1$. Thus, $M k$ satisfies the Assumptions in \ref{Ass:MainAssumptions}. 
    
    Furthermore, we have
    \begin{equation*}
        m_M(dy) = e^{(1-M)\int^y_{c_0} \frac{2k(z)}{\sigma^2(z)}dz } m_1(dy),
    \end{equation*}
    where $m_{M}$ is the speed measure corresponding to the drift $\mu-Mk$ and $m_1$ corresponding to drift $\mu-k$. Thus, $\int_0^{\infty} m_1(dy) < \infty$ implies that $\int_0^{\infty} m_M(dy) < \infty$.

    Let $x_0 > x^*$ be the unique point such that $\mu(x_0)=\mu(0)$.
    Then $\mu(y)-\mu(0) > 0$ on $(0,x_0)$ and $\mu(y)-\mu(0) < 0$ on $(x_0, \infty)$. Since the mapping 
    \begin{equation*}
        y \mapsto e^{(1-M)\int^y_{c_0} \frac{2k(z)}{\sigma^2(z)}dz }
    \end{equation*}
    is nonincreasing, we find that
    \begin{align*}
        \int_0^\infty(\mu(y)-\mu(0))m_M(dy) & =\int_0^\infty(\mu(y)-\mu(0))e^{(1-M)\int^y _{c_0}\frac{2k(z)}{\sigma^2(z)}dz }m_1(dy) \\
        & \geq e^{(1-M)\int^{x_0}_{c_0} \frac{2k(z)}{\sigma^2(z)}dz }\int_0^\infty(\mu(y)-\mu(0))m_1(dy).
    \end{align*}
    As the last integral is strictly positive by the assumption $\mu(0) < \overline\mu_1$, we get
    \begin{equation*}
        \int_0^\infty(\mu(y)-\mu(0))m_M(dy)>0,
    \end{equation*}
    which implies $\mu(0)<\overline\mu_M$. Hence, $Mk$ also satisfies Assumption~\ref{hit 0}.

    Define
    \begin{equation*}
        H_M(x):=\frac{\int_x^\infty \mu(y)m_{M}(dy)}{\int_x^\infty m_{M}(dy)},
    \end{equation*}
    where $m_{M}$ is the speed measure corresponding to the drift $\mu-Mk$, and note that for $\delta>0$ we have
    \begin{eqnarray*}
        H_M(x) &=& \frac{\int_x^\infty \mu(y) e^{-M\int_c^y \frac{2k(z)}{\sigma^2(z)}dz} m(dy)}
        {\int_x^\infty e^{-M\int_c^y \frac{2k(z)}{\sigma^2(z)}dz}m(dy)}\\
        &=& \frac{\int_x^{x+\delta} \mu(y) e^{-M\int_{x+\delta}^y \frac{2k(z)}{\sigma^2(z)}dz} m(dy) + \int_{x+\delta}^\infty \mu(y) e^{-M\int_{x+\delta}^y \frac{2k(z)}{\sigma^2(z)}dz} m(dy)}
        {\int_x^{x+\delta} e^{-M\int_{x+\delta}^y \frac{2k(z)}{\sigma^2(z)}dz}m(dy) + \int_{x+\delta}^\infty e^{-M\int_{x+\delta}^y \frac{2k(z)}{\sigma^2(z)}dz}m(dy)}.
    \end{eqnarray*}
    Now, the two indefinite integrals both tend to 0 as $M\to\infty$, whereas the integral 
    $\\ \int_x^{x+\delta} e^{-M\int_{x+\delta}^y \frac{2k(z)}{\sigma^2(z)}dz}m(dy)$ explodes. It follows that
    \[\inf_{y\in[x,x+\delta]}\mu(y)\leq \liminf_{M\to\infty} H_M(x)\leq \limsup_{M\to\infty} H_M(x)\leq \sup_{y\in[x,x+\delta]}\mu(y).\]
    Since $\delta>0$ is arbitrary and $\mu$ is continuous, we thus have 
    \[\lim_{M\to\infty}H_M(x)=\mu(x)\]
    for all $x>0$ with $k(x)>0$.
        
    Next, recall that $b^*_M$ and $x'$ are the unique roots of $G(x)-H_M(x)$ and $G(x)-\mu(x)$, respectively, with $b^*_M\leq x'$ and 
    $G(x)<\mu(x)$ for $x<x'$. Since
    \[G(x)-H_M(x) \to G(x)-\mu(x)\] 
    as $M \to \infty$, the roots much converge, and so
    \begin{equation*}
        \lim_{M \to \infty }b_M^* = x'.
    \end{equation*}
    
    Finally, since $\beta_M^*=G(b_M^*)$ and $G$ is continuous, we obtain
    \begin{equation*}
        \beta_M^* = G(b_M^*) \xrightarrow{M\to \infty} G(x') = \mu(x').
    \end{equation*}
\end{proof}

\section{Illustration}\label{Sec:Numerical}

Consider a population whose biomass evolves according to
\begin{equation*}
    dX_t = \big(c + rX_t(1-X_t/K)\big)dt + \sigma X_t\,dW_t, \quad X_0=x>0,
\end{equation*}
where $c>0$ represents recruitment or immigration, $r>0$ is the intrinsic reproduction rate, $K>0$ scales the carrying capacity, and $\sigma>0$ measures environmental variability. For this model, the density of the speed measure is
\begin{align*}
    m'(x) = x^{-2r /\sigma^2} \exp{\left( -\frac{2r}{\sigma^2 K}x - \frac{2c}{\sigma^2 x} \right)}
\end{align*}
and the drift $\mu(x)=c+r x(1-x/K)$ attains its maximum at $x^*=K/2$,
and is increasing on $(0,x^*)$ and decreasing on $(x^*,\infty)$. Thus $\mu$ satisfies the conditions in Assumption~\ref{Ass:MainAssumptions}.

Since the first-order condition for the optimal threshold \eqref{Eq:b} cannot be solved explicitly (it includes special functions), we illustrate the results numerically. For this we take the values $c=0.25, r=1.2, K=1.0, \sigma=0.55$.

We now demonstrate following our general results that fixed population dynamics can lead to two qualitatively different optimal harvesting strategies, depending only on the bound on the harvesting rate. For this, we study the harvesting bound
\begin{equation*}
    k_M(x):=M \left( \frac{1}{10} + \frac{x}{20} \right),
\end{equation*}
where $M \geq 1$ is a constant.

We have shown in previous sections that the key condition determining the optimal policy is whether $\bar{\mu} < \mu(0)$ (leading to always extracting) or $\bar{\mu} > \mu(0)$ (leading to a positive threshold). Figure \ref{Fig:Thresholds} demonstrates this transition: the boundary is zero for small $M$ and becomes strictly positive at $M \approx 1.73$. These results are intuitive, since when $M$ is small enough, the population is effectively protected from heavy extractions and still benefits from the positive immigration $c$, leading to fast recovery for the population. In other words, it is not harmful for the population to extract all the time.

The same figure also highlights our comparative statics from Section \ref{Sec:Comparative}. Firstly, we can see that by increasing $M$ the boundary $b^*$ is increasing, agreeing with our general result in Proposition~\ref{Prop: k1k2ordering}. Secondly, following Proposition~\ref{Prop: limitharvesting}, we see that the threshold $b^*$ ultimately converges towards $x'$, which corresponds to the optimal reflection when harvesting is not restricted.

\begin{figure}
    \centering
    \includegraphics[width=0.85\linewidth]{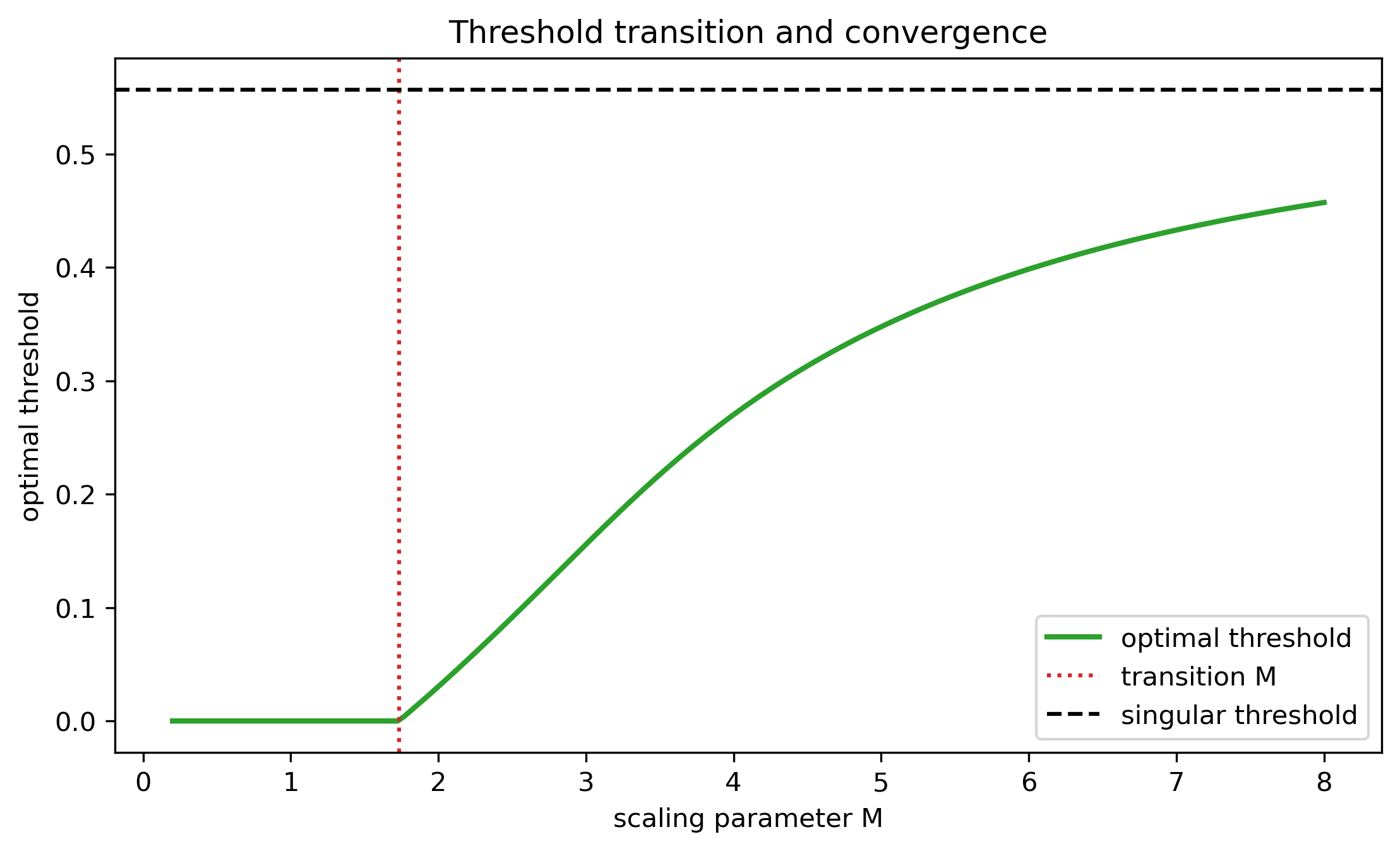}
    \caption{The optimal threshold under the harvesting bound $k_M(x)$ and its convergence to the singular reflection threshold $x'$.}
    \label{Fig:Thresholds}
\end{figure}

Figure~\ref{Fig:Densities} shows the biomass distributions under the optimal strategies for $M=1$ and $M=4$, together with the singular reflection limit. In the first case ($M=1$) the harvesting bound is strict enough to make the optimal strategy to be always extracting, and the density is smooth. For $M=4$ the optimal density includes a kink at $b^*_M$ due to the drift change above the thresholds. In the extreme case $M = \infty$ the population cannot exceed the optimal reflecting boundary $x'$.

\begin{figure}
    \centering
    \includegraphics[width=0.85\linewidth]{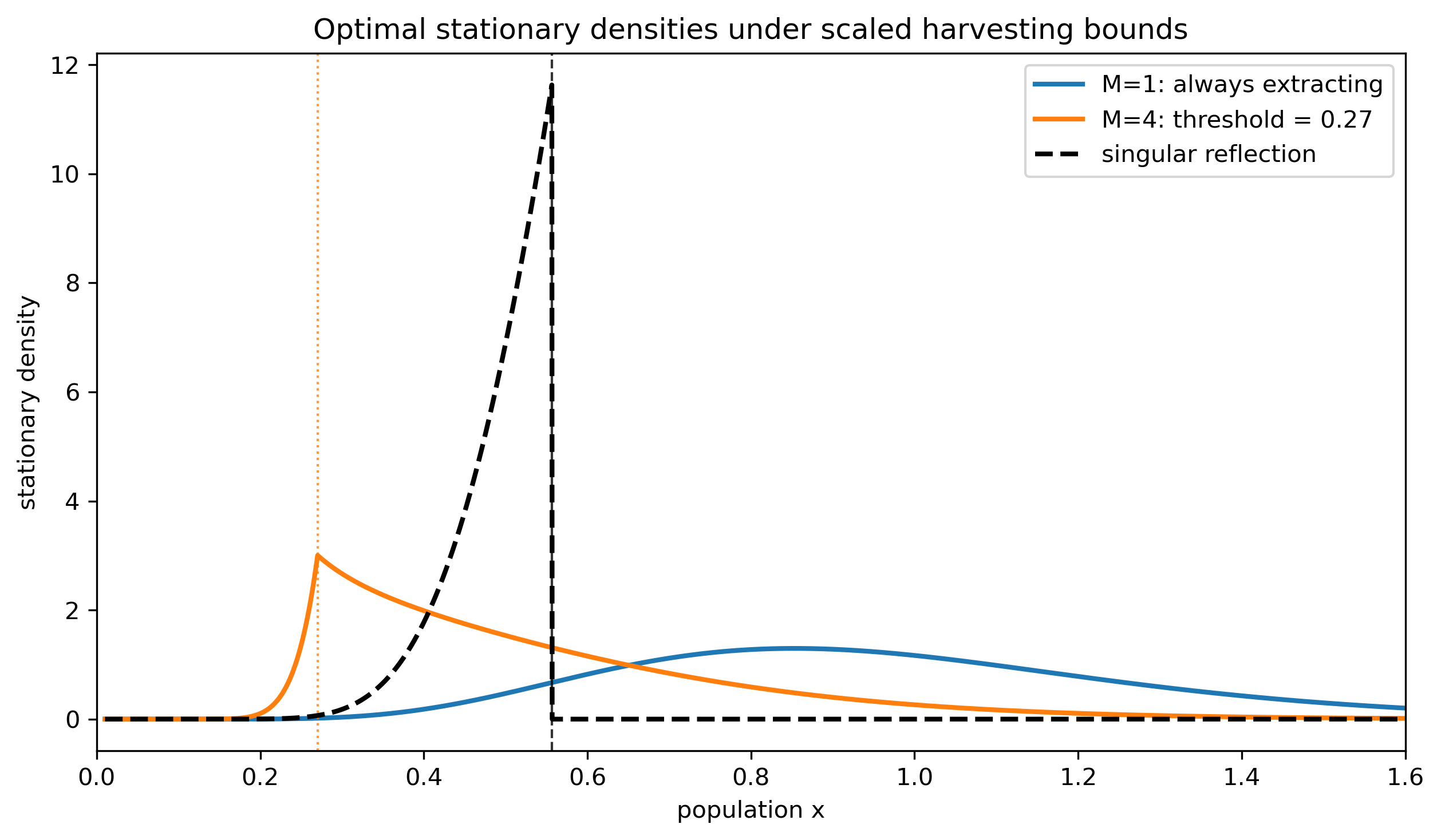}
    \caption{Stationary biomass densities under optimal policies for the harvesting bound $k_M(x)$. The dotted vertical line indicates the positive bounded-rate threshold for $M=4$, while the dashed black line represents the singular reflection limit.}
    \label{Fig:Densities}
\end{figure}

\section*{Acknowledgements}

Erik Ekström gratefully acknowledges support from the Swedish Research Council.
Funding for Dante Mata in support of this work was provided by a FRQNT Postdoctoral Fellowship (369449) from the Fonds de Recherche du Qu\'ebec. In addition, Dante Mata stayed in the Department of Mathematics at Uppsala University and received support from the research environment there.
The work of Harto Saarinen was supported by the OP Research Foundation Grant 20240114.

\newpage

\bibliographystyle{siam}
\bibliography{literature}

\end{document}